\documentclass[a4paper,12pt]{amsart}
\usepackage[letterpaper, margin=1in]{geometry}
\usepackage{latexsym}
\usepackage{amssymb}
\usepackage{amsmath}
\usepackage{amsfonts}

            \DeclareFontFamily{U}{wncy}{}
            \DeclareFontShape{U}{wncy}{m}{n}{%
               <5>wncyr5%
               <6>wncyr6%
               <7>wncyr7%
               <8>wncyr8%
               <9>wncyr9%
               <10>wncyr10%
               <11>wncyr10%
               <12>wncyr6%
               <14>wncyr7%
               <17>wncyr8%
               <20>wncyr10%
               <25>wncyr10}{}

\newtheorem{thm}{Theorem}[section]

\newtheorem{lem}[thm]{Lemma}

\newtheorem{cor}[thm]{Corollary}

\theoremstyle{definition}
\newtheorem*{dfn}{Definition}
\newtheorem*{remark}{Remark}

\newtheorem*{acknowledgment}{Acknowledgment}

\newcommand{\z}{\mathbb Z}
\newcommand{\q}{{\mathbb Q}}

\newcommand{\F}{\mathbb F}
\def\op{\operatorname}

\def\al{\alpha}

\def\diso{\lower.4ex\hbox{$\downarrow$}\raise.4ex\hbox{\mbox{\scriptsize $\wr$}}}

\def\fph{\mathbb{F}_{\ph}}

\def\al{\alpha}

\def\iso{\,\lower .6ex\hbox{$\stackrel{\lra}{\mbox{\tiny $\sim\,$}}$}\,}

\def\la{\lambda}

\def\lg{l\raise.6ex\hbox to.2em{\hss.\hss}l}
\def\lra{\longrightarrow}

\def\md#1{\ \mbox{\rm(mod }{#1})}

\def\orb{\hbox to  .3em{$\backslash$}\backslash}

\def\p{\mathfrak{p}}

\def\ph{\phi}

\def\K{{\mathbb{K}}}

\def\rd{\op{red}}

\def\t{\theta}

\newcounter{cs}
\stepcounter{cs}
\newcommand{\casos}{\begin{itemize}}
\newcommand{\fcasos}{\end{itemize}\setcounter{cs}{1}}

\newfont{\tit}{cmr12 scaled \magstep3}

\begin{document}

\title[Extensions of a Disc. Valuation in a Number Field]{On the Extensions of a Discrete Valuation in a Number Field}

\author{Abdulaziz Deajim}
\address{Department of Mathematics, King Khalid University,
P.O. Box 9004, Abha, Saudi Arabia} \email{deajim@gmail.com}

\author{Lhoussain El Fadil}
\address{Poly-disciplinary Faculty of Ouarzazat, Ibn Zohr University, P.O. Box 639, Ouarzazat, Morocco} \address{ Department of Mathematics, King Khalid University,
P.O. Box 9004, Abha, Saudi Arabia } \email{lhouelfadil2@gmail.com}

\date{\today}

\begin {abstract}
Let $K$ be a number field defined by a monic irreducible polynomial \linebreak $F(X) \in \z[X]$, $p$ a fixed rational prime, and $\nu_p$ the discrete valuation associated to $p$. Assume that $\overline{F}(X)$ factors modulo $p$ into the product of powers of $r$ distinct monic irreducible polynomials. We present in this paper a condition, weaker than the known ones, which guarantees the existence of exactly $r$ valuations of $K$ extending $\nu_p$. We further specify the ramification indices and residue degrees of these extended valuations in such a way that generalizes the known estimates. Some useful remarks and computational examples are also given to highlight some improvements due to our result.

\end {abstract}
\maketitle

\section{{\bf Introduction}}\label{intro}

Let $K$ be a number field, $\mathbb{Z}_K$ its ring of integers, $\alpha \in \mathbb{Z}_K$ a primitive element of $K$, $\mbox{ind}(\alpha)=[\mathbb{Z}_K : \mathbb{Z}[\alpha]]$ (the index of $\mathbb{Z}[\alpha]$ in $\mathbb{Z}_K$), and $F(X)\in \mathbb{Z}[X]$ the minimal polynomial of $\alpha$ over $\mathbb{Q}$. One of the most important problems in algebraic number theory is determining the prime ideal factorization in $\z_K$ of a rational prime $p$. Due to a well-known Theorem of Hensel \cite{Hen}, this problem is directly related to the factorization of $F(X)$ in $\q_p[X]$, which in turn is related to the extensions of $\nu_p$ to $K$ (see Lemma \ref{tool}).

If $p$ does not divide the index $\mbox{ind}(\alpha)$, then a theorem of Kummer (see \cite[Theorem 7.4]{Jan}) indicates that the factorization of $p\z_K$ can be derived directly from the decomposition of $\overline{F}(X)$ modulo $p$; namely, if $\overline{F}(X) = \prod_{i=1}^r \overline{\phi}_i(X)^{l_i}$ is the factorization of $\overline{F}(X)$ into the product of powers of distinct monic irreducible polynomials $\overline{\phi}_i(X)$ modulo $p$, then \linebreak $p\z_K=\prod_{i=1}^r\p_i^{l_i}$, where $\p_i=(p, \ph_i(\alpha))$ with ramification index $e(\p_i/p)=l_i$ and residue degree  $f(\p_i/p)=\mbox{deg}(\ph_i)$. So, in particular, there are exactly $r$ valuations of $K$ extending $\nu_p$ (by Hensel's Theorem and Lemma \ref{tool}).

In 1878, R. Dedekind \cite{Ded} gave a criterion that tests when $p$ does, or does not, divide $\mbox{ind}(\alpha)$. In \cite{KK1}, S. Khanduja and M. Kumar (K-K) proved that the condition "$p$ does not divide $\mbox{ind}(\alpha)$" is necessary for the existence of exactly $r$ distinct prime ideals of $\mathbb{Z}_K$ lying above $p$ (and, thus, the existence of exactly $r$ distinct valuations of $K$ extending $\nu_p$) with generators, ramification indices, and residue degrees all as above. They went on in \cite{KK2} to ask whether it is possible to find a weaker condition that guarantees the existence of exactly $r$ valuations of $K$ extending $\nu_p$ with ramification indices and residue degrees all as above. In the same paper \cite[Theorem 1.1]{KK2}, they successfully gave a weaker sufficient condition, a result which we state in Section \ref{main results}.

The main goal of this paper is to give an improvement of \cite[Theorem 1.1]{KK2} (in the context of number fields) in the form of our main result: Theorem \ref{Main 1}. In Section \ref{main results}, we introduce some concepts and notations relevant to the work. Using these notations and terminologies, we also give an equivalent version of K-K's result followed by a statement of our main theorem. In Section \ref{lemmas}, we present some lemmas, some of which are interesting in their own, to help in proving the main theorem. In Section \ref{proofs}, we tackle the proof of the main theorem. We follow that with some quite useful remarks in Subsection \ref{remarks}. In Subsection \ref{examples}, we give some examples to shed some light on the improvements caused by our results and their consequences.

\section{Notations and Main Result}\label{main results}

Fix a rational prime $p$. Let $\mathbb{Z}_p$ be the ring of $p$-adic integers and $\nu_p$ the discrete \linebreak$p$-adic valuation associated to $p$. Extend $\nu_p$ to the ring $\mathbb{Z}_p[X]$ by defining the $p$-adic valuation of a polynomial in $\mathbb{Z}_p[X]$ as the minimum of the $p$-adic valuation of its coefficients. For a polynomial $P(X) \in \mathbb{Z}_p[X]$, denote by $\overline{P}(X)$ the reduction of $P(X)$ modulo $p$. Let \linebreak $\phi(X) \in \mathbb{Z}_p[X]$ be a fixed monic polynomial whose reduction modulo $p$ is irreducible and \linebreak $\mathbb{F}_\phi := \mathbb{Z}_p[X]/(p, \phi(X))\;(\cong \F_{p^{\Small{\mbox{deg}}(\phi)}})$ be the finite field defined by $p$ and $\phi(X)$. \linebreak Let $\mbox{red}:\mathbb{Z}_p[X] \to \mathbb{F}_\phi$ the canonical projection. For a monic polynomial $F(X) \in \mathbb{Z}_p[X]$, let
$$F(X) = a_0(X) \phi(X)^l + a_1(X) \phi(X)^{l-1} +\dots + a_{l-1}(X) \phi(X) + a_l(X)$$
be the $\phi$-adic expansion of $F(X)$ in $\mathbb{Z}_p[X]$. So $a_i(X) \in \mathbb{Z}_p[X]$ and $\mbox{deg}(a_i(X)) < \mbox{deg}(\phi(X))$ for all $i$. Set $s_F=s=\mbox{max}\{0\leq i \leq l \,|\, \nu_p(a_i(X))=0\}$. We then have the following two-case setup:\\

\noindent {\underline{\bf Case 1. $\overline{\phi}(X)$ divides $\overline{F}(X)$:}} In this case $s<l$. 
Set $\lambda_F =\lambda= \nu_p(a_l(X))/(l-s)$. Note that $\lambda$ is the slope of the line joining the two points $(s,0)$ and $(l, \nu_p(a_l(X)))$ in the Euclidean plane. We say here that $F$ satisfies the $L_{\phi, \lambda}$ property if $\nu_p(a_{i+s}(X))/i \geq \lambda$ for every $i=1, \dots, l-s$.
Let $d=\mbox{gcd}(\nu_p(a_l(X)), l-s)$, $e=(l-s)/d$, and $h=\nu_p(a_l(X))/d$. Notice that $e$ and $h$ are the unique positive coprime integers such that $\lambda=h/e$. If $F$ satisfies the $L_{\ph, \la}$ property, then set $t_j^*=\mbox{red}(a_{j+s}(X)/p^{\lfloor j\la \rfloor})$ for $j=0, \dots, l-s$. Since $\nu_p(a_j(X)) \geq j\la$ for all such $j$, it follows that if $j\la \not\in \z$, then $t_j^* =0$. Noticing that $j\la \in \z$ if and only if $j \in e \z$, we set $t_i=\mbox{red}(a_{ie+s}(X)/p^{ih})$ for $i=0, \dots, d$ and define the $\ph$-residual polynomial attached to $F$ by $F_\phi(Y) = \sum_{i=0}^d t_i Y^{d-i} \in \mathbb{F}_\phi[Y]$. Note that $t_i\neq 0$ if and only if $\nu_p(a_{ie+s}(X))=ih$. In particular, $t_0$ is never 0 and, thus, $\mbox{deg}(F_\phi(Y))=d$.\\

\noindent{\underline{\bf Case 2. $\overline{\phi}(X)$ does not divide $\overline{F}(X)$:}} In this case $s=l$. Set $\lambda_F=\lambda =0$, $d=l$, and $e=1$. Here we say that $F$ always satisfies the $L_{\phi, \lambda}$ property. For $i=0, \dots, l$, set $t_i =\mbox{red}(a_i(X))$ and define the residual polynomial attached to $F$ by $F_\phi(Y) = \sum_{i=0}^l t_i Y^{l-i} \in \mathbb{F}_\phi[Y]$.\\

We now set up the notations and assumptions for our main theorem. Let $F(X) \in \mathbb{Z}[X]$ be a monic irreducible polynomial, $\alpha$ a complex root of $F$, $K =\mathbb{Q}(\alpha)$, $\mathbb{Z}_K$ the ring of integers of $K$, and $\overline{F}(X) = \prod_{i=1}^r \overline{\phi}_i^{l_i}(X)$ modulo $p$, where $\overline{\phi}_i(X)$ are all monic and irreducible. For each $i=1, \dots, r$, let $F(X) = \sum_{j=0}^{L_i} a_{i,j}(X) \phi_i^{L_i - j}(X)$ be the $\phi_i$-adic expansion of $F(X)$ in $\mathbb{Z}[X]$, where $\phi_i(X) \in \mathbb{Z}[X]$ is a monic lift of $\overline{\phi_i}(X)$. Set $s_{F,i}=\mbox{max}\{j=0, \dots, L_i \,|\, \nu_p(a_{i,j}(X)) =0\}$, $\lambda_i = \nu_p(a_{i,L_i}(X))/(L_i -s_{F,i})$, $m_i=\mbox{deg}(\phi_i(X))$, and $d_i =\mbox{gcd}(\nu_p(a_{i,L_i}(X)), L_i - s_{F,i})$. We see that $s_{F,i} < L_i$ for each $i$ as $\overline{\phi_i}(X)$ divides $\overline{F}(X)$. So whenever the property $L_{\phi, \lambda}$ is mentioned in the sequel, we are referring only to case 1 above. It should not cause any confusion, though, that instead of using the notation $s_{F,i}$, we will use $s_F$ (or even $s$) when $\phi_i$ (or $\phi_i$ and $F$) are clear in the context.

Given the above notations, assumptions, and terminology, K-K's result \linebreak(\cite[Theorem 1.1]{KK2}), which we seek to improve here, can be reformulated (in our context of number fields) as follows:

{\it If for every $i=1, \dots, r$, $F(X)$ satisfies the $L_{\phi_i, \lambda_i}$ property and $d_i=1$ (or, equivalently, $F_{\phi_i}(Y)$ is linear), then there are exactly $r$ valuations $w_1, \dots, w_r$ of $K$ extending $\nu_p$. Moreover, the respective ramification indices and residue degrees of these valuations are $e(w_i)=l_i$ and $f(w_i)=m_i$}.

Our main result, Theorem \ref{Main 1} below, relaxes the linearity of $F_{\phi_i}(Y)$ and, further, specifies the ramification indices and residue degrees of the valuations $w_i$'s in such a way that generalizes K-K's result.

\begin{thm}\label{Main 1}

Keep the notations, assumptions, and terminology as above. For $i=1, \dots, r$, suppose that either

(i) $l_i=1$, or

(ii) $l_i \geq 2$, $F(X)$ satisfies the $L_{\phi_i, \lambda_i}$ property, and $F_{\phi_i}(Y)$ is irreducible.\\
Then there are exactly $r$ valuations $w_1, \dots, w_r$ of $K$ extending  $\nu_p$. Furthermore, for each \linebreak $i=1, \dots, r$, the ramification index of $w_i$  is $e(w_i)=  l_i/d_i$ and its residue degree is $\displaystyle{f(w_i) = m_i d_i}$.

\end{thm}




In remark (1-a) of Subsection \ref{remarks}, we emphasize the necessity of $F$ satisfying the $L_{\phi_i, \lambda_i}$ property above for all $i=1, \dots, r$ in order to get exactly $r$ valuations $w_1, \dots, w_r$ of $K$ extending  $\nu_p$. Besides, Theorem \ref{Main 1} helps in giving a new, easier proof of K-K's result of \cite{KK1} (see remark (3) of Subsection \ref{remarks}).

\section{Lemmas}\label{lemmas}

Recall that if $(L, \nu)$ is a valued field and $(\hat{L}, \hat{\nu})$ is its $\nu$-adic completion, then the separable closure $L_h$ of $L$ in $\hat{L}$ is called a Henselization of $L$ with respect to $\nu$ (see \cite[Chapter II - $\S$6]{Neu}). Denote by $\nu_h$ the restriction to $L_h$ of the unique extension of $\nu$ to the algebraic closure of $L_h$.

\begin{lem}\label{tool}{\normalfont (}\cite[Thoerem 2.D.]{KK2}{\normalfont )} Let $(L,\nu)$ be a valued field, $L(\t)$ an algebraic extension of $L$, $F(X)$ the minimal polynomial of $\t$ over $L$, and $L_h$ the Henselization of $L$ in $\hat{L}$. Then, the valuations $w_1,\dots, w_t$ of $L(\t)$ extending $\nu$ are in one-to-one correspondence with the irreducible factors $F_1(X),\dots, F_t(X)$ of $F(X)$ in  $L_h[X]$. Moreover, the valuations $w_i$ are defined precisely by $$w_i(\sum_j a_j\t^j)=\nu_h(\sum_j a_j\t_i^j)$$ for any root $\t_i$ of $F_i(X)$ and $a_j \in L$.
\end{lem}

We note on passing that since $\q$ is of characteristic zero, the factorizations in $\q_h[X]$ and in $\q_p[X]$ of a polynomial with integer coefficients are the same.

Throughout this section, let $p\in \z$ be a fixed prime, $\z_p$ the ring of $p$-adic integers, $f(X), g(X), \phi(X) \in \mathbb{Z}_p[X]$ be all monic, and $$f(X)=\sum_{i=0}^{\l_1} a_i(X) \phi(X)^{\l_1 -i}\;\; \mbox{and} \;\; g(X) =\sum_{i=0}^{\l_2} b_i(X) \phi(X)^{\l_2 -i}$$ the $\phi$-adic expansions of  $f$ and $g$ in $\mathbb{Z}_p[X]$.

The following lemma gives a procedure to compute the $\phi$-adic expansion of the product $fg(X)$ .

\begin{lem}\label{expprod}
The $\phi$-adic expansion of the product $fg(X)$ in $\z_p[X]$ is given by the following steps:

(1) Let $\l= \l_1 + l_2$ and, for each $k = 0, \dots, \l$, set
$c_k(X) = \sum_{i=0}^k a_i(X) b_{k-i}(X)$, where $a_i(X)=0$ for $i>l_1$ and $b_j(X)=0$ for $j>l_2$.

(2) For each $k=0, \dots, \l$, use the Euclidean algorithm to find
$Q_k(X)$ and $R_k(X)$ in $\mathbb{Z}_p[X]$ such that $c_k(X) =
\phi(X) Q_k(X) + R_k(X)$, where $R_k(X)=0$ or $\mbox{deg}(R_k(X))< \mbox{deg}(\ph(X))$.

(3) Set $A_{\l}(X) = R_{\l}(X)$, and for each $k=0, \dots, \l-1$,
set $A_k(X)= Q_{k+1}(X) + R_k(X)$. Furthermore, if
$\mbox{deg}(a_0b_0(X))< \mbox{deg}(\phi(X))$, then set $A_{-1}(X)
=0$; otherwise set $\displaystyle{A_{-1}(X) = Q_0(X)}$.

(4) Conclude that $fg(X) = \sum_{k=-1}^{\l} A_k(X) \phi(X)^{\l-k}$
is the $\phi$-adic expansion of $fg(X)$.

\end{lem}

\begin{proof}
It is clear that $fg(X)=c_0(X)\ph(X)^{l}+\dots+c_{\l}(X)$. It is, however, not guaranteed that $\mbox{deg}(c_k(X))<\mbox{deg}(\phi(X))$ for $k=0, \dots, l$. So it suffices to
make the Euclidean division $c_k(X)=\ph(X)Q_k(X)+R_k(X)$, keep the
remainder for the $k^{th}$ coefficient and transfer the quotient
to the $(k-1)^{st}$ coefficient.
\end{proof}

\begin{remark}
Note that for $k=0, \dots, l$, $\nu_p(c_k(X)) \geq \mbox{min}\{\nu_p(Q_k(X)), \nu_p(R_k(X))\}$. But as $Q_k(X)$ and $R_k(X)$ are respectively the quotient and the remainder resulting from the Euclidian division of $c_k(X)$ by a monic polynomial, we indeed have $$\nu_p(c_k(X)) = \mbox{min}\{\nu_p(Q_k(X)), \nu_p(R_k(X))\}.$$
\end{remark}

\begin{lem}\label{slopezero}
Assume that $\overline{\phi}(X)$ is irreducible, $\overline{\ph}(X)$ does not divide $\overline{g}(X)$ modulo $p$, and $\overline{f}(X)$ is a power of $\overline{\ph}(X)$ modulo $p$. Then $f(X)$ satisfies the $L_{\phi, \lambda_f}$ property if and only if $fg(X)$ satisfies the $L_{\phi, \lambda_f}$ property. If any of these equivalent conditions holds, then $(fg)_\phi(Y) = c f_\phi(Y)$ for some nonzero constant $c \in \fph$.
\end{lem}

\begin{proof}
Following the notations of Lemma \ref{expprod}, consider $c_k$, $Q_k$, $R_k$, and $A_k$. As $\overline{f}(X)$ is a power of $\overline{\phi}(X)$, $a_0(X)=1$, $\nu_p(a_i(X))>0$ for all $i=1, \dots, l-1$, and $\overline{f}(X)=\overline{\phi}(X)^{l_1}$ modulo $p$. Thus $s_f=0$ and $\lambda_f=\nu_p(a_{l_1}(X))/l_1$. On the other hand, since $\overline{\phi}(X)$ does not divide $\overline{g}(X)$, $\nu_p(b_{l_2}(X))=0$.

Assume now that $f$ satisfies the $L_{\phi, \lambda_f}$ property. So $\nu_p(a_i(X))\ge i\lambda_f$ for $i=1,\dots, l_1$. Set $c_{l_2}(X)=a_0(X) b_{l_2}(X)+\sum_{i=1}^{l_2}a_i(X) b_{l_2-i}(X)$ where, for $i>l_1$ (if any), $a_i(X)=0$ (and hence $\nu_p(a_i(X))=\infty$ for all such $i$). Then $\nu_p(c_{l_2}(X))=0$ since $\displaystyle{\nu_p(a_0(X)b_{l_2}(X))=0}$ and $\nu_p(a_i(X)b_{l_2 -i}(X))\geq \nu_p(a_i(X)) \geq i\lambda_f>0$ for $i=1, \dots, l_2$. Moreover, $\overline{\phi}(X)\nmid \overline{a_0}(X)\overline{b_{l_2}}(X)$ because $\overline{\phi}(X)$ is irreducible and both $\mbox{deg}(\overline{a_0}(X)), \, \mbox{deg}(\overline{b_{l_2}}(X))<\mbox{deg}(\overline{\phi}(X))$. Since \linebreak $\overline{c_{l_2}}(X) = \overline{a_0}(X)\overline{b_{l_2}}(X)$ modulo $p$, it thus follows that $\overline{R_{l_2}}(X) \neq \overline{0}$ modulo $p$. Thus, \linebreak $\nu_p(R_{l_2}(X))=0$, $\nu_p(A_{l_2}(X))=0$, and $s_{fg}=l_2$.

For $i=1, \dots, l_1$, we can write $c_{l_2 +i}(X)=\sum_{j=0}^{l_2} a_{i+j}(X)b_{l_2 -j}(X)$. As, for $j=1, \dots, l_2$, $\nu_p(a_{i+j}(X))\geq (i+j)\lambda_f$, it follows that $\nu_p(c_{l_2 +i}(X))\geq i\lambda_f$, $\nu_p(Q_{l_2+i}(X))\geq i\lambda_f$, and $\nu_p(R_{l_2+i}(X))\geq i\lambda_f$. Therefore, $\nu_p(A_{l_2+i}(X))\geq i\lambda_f > 0$ and, thus, $s_{fg}=l_2$. To prove that $fg$ satisfies the $L_{\ph, \la_f}$ property, it remains to show that $\la_f=\nu_p(A_l)/l_1$. Note that $c_l(X)=a_{l_1}(X)b_{l_2}(X)=\phi(X)Q_l(X)+A_l(X)$. Since $$\nu_p(Q_l(X))\geq l_1 \lambda_f \;\; \mbox{and} \;\; \nu_p(c_l(X))=\nu_p(a_{l_1}(X)b_{l_2}(X))=\nu_p(a_{l_1}(X))=l_1 \lambda_f,$$ we have the equality $c_l(X)/p^{l_1 \lambda_f}= \phi(X) \, Q_l(X)/p^{l_1 \lambda_f}+ A_l(X)/p^{l_1 \lambda_f}$ in $\mathbb{Z}_p[X]$. So, \linebreak $\overline{A_l}(X)/p^{l_1 \lambda_f}\neq \overline{0}$ modulo $p$ as $\nu_p(c_l(X)/p^{l_1\lambda_f})=0$ and $\overline{\phi}(X) \nmid \overline{c_l}(X)/p^{l_1 \lambda_f}$. Hence, \linebreak $\nu_p(A_l(X))=l_1\lambda_f$ and, therefore, $fg$ satisfies the $L_{\phi, \lambda_f}$ property.

Conversely, let $\lambda=\nu_p(a_{l_1}(X))/l_1$, and assume that $fg(X)$ satisfies the $L_{\phi,\lambda}$ property. As $a_0(X)=1$, $c_0(X)=a_0(X)b_0(X)=b_0(X)$, and $\mbox{deg}(b_0(X)) < \mbox{deg}(\phi(X))$, the quotient $Q_0(X)=0$. So $A_{-1}(X)= Q_0(X)=0$. Since $c_{l_2}(X)=a_0(X)b_{l_2}(X)+\sum_{i=1}^{l_2} a_i(X)b_{l_2-i}(X)$, it follows that $\nu_p(\sum_{i=1}^{l_2}a_i(X)b_{l_2-i}(X))\geq 1$ (as $\nu_p(a_i(X))>0$ for $i=1,\dots, l_2$), and \linebreak $\nu_p(a_0(X))+\nu_p(b_{l_2}(X))=0$, $\nu_p(c_{l_2}(X))=0$. It also follows that $\nu_p(A_{l_2}(X))=0$. On the other hand, for $i\geq 1$, $\nu_p(c_{l_2+i}(X))\geq 1$ as $c_{l_2+i}(X)=\sum_{j=i}^{l_2+i} a_j(X) b_{l_2+i-j}(X)$ and $\nu_p(a_j(X))\geq 1$ for $j\geq i \geq 1$. It thus follows that $\nu_p(A_{l_2+i}(X)\geq 1$ for $i\geq 1$. Hence, $$s_{fg}=\mbox{max}\{0\leq i \leq l\,|\,\nu_p(A_i(X))=0\}=l_2.$$ Assume that $\nu_p(a_i(X))<i\la$ for some $1\leq i\leq l_1$ such that; let $i_0$ be the maximum such index. Consider
$\displaystyle{c_{\l_2+i_0}(X)=a_{i_0}(X)b_{\l_2}(X)+\sum_{i\ge i_0+1}a_i(X)b_{\l_2-i}(X)}$. Then $\displaystyle{\nu_p(c_{\l_2+i_0}(X))<i_0\la}$, $\nu_p(A_{\l_2+i_0}(X))<i_0\la$ and $fg(X)$ does not satisfy the $L_{\phi,\la}$ property. Hence, $\nu_p(a_i(X))\geq i\lambda$ for all $i=1, \dots, l_1$, and $f(X)$ satisfies the $L_{\phi,\la}$ property.

Now assume that one of the above conditions holds. Let $d=\mbox{gcd}(\nu_p(a_{l_1}(X)),l_1)$, $e=l_1/d$, $h=\nu_p(a_{l_1}(X))/d$, and $\la=\nu_p(a_{l_1}(X))/l_1=h/e$. For every $i=0,\dots,d$, $$A_{\l_2+ie}(X)/p^{ih}\equiv
R_{\l_2+ie}(X)/p^{ih}+Q_{\l_2+ie+1}(X)/p^{ih}\equiv
R_{\l_2+ie}(X)/p^{ih} \;\;(\mbox{mod} \,p)$$ (because $\nu_p(Q_{l_2+ie+1}(X))>ie\lambda=ih$). Moreover, $$R_{\l_2+ie}(X)/p^{ih}\equiv
c_{\l_2+ie}(X)/p^{ih}-\ph(X)Q_{\l_2+ie}(X)/p^{ih}
\equiv c_{\l_2+ie}(X)/p^{ih}\;\;(\mbox{mod}\,(p,\ph)).$$ So, $A_{\l_2+ie}(X)/p^{ih}\equiv
c_{\l_2+ie}(X)/p^{ih}$ $(\mbox{mod}\,(p,\ph))$. Using the expression
of $c_{\l_2+ie}(X)$, we have $c_{\l_2+ie}(X)/p^{ih}\equiv
b_{\l_2}(X)a_{ie}(X)/p^{ih}$ $(\mbox{mod}\, p)$ and
$A_{\l_2+ie}(X)/p^{ih} \equiv b_{\l_2} a_{ie}(X)/p^{ih}$ $(\mbox{mod}\,(p,\ph))$. Therefore, $(fg)_{\ph}(Y)=c f_{\ph}(Y)$, where $c=\rd(b_{\l_2}(X))$.
\end{proof}

\begin{lem}\label{prod}
If $\overline{fg}(X)$ is a power of $\overline{\phi}(X)$ modulo $p$
and $\overline{\phi}(X)$ is irreducible, then $fg$ satisfies the
$L_{\phi, \lambda}$ property if and only if both $f$ and $g$
satisfy the same $L_{\phi, \lambda}$ property, where $\lambda=\lambda_{fg}=\lambda_f=\lambda_g$. Moreover,
$(fg)_\phi (Y)= f_\phi(Y) g_\phi(Y)$ if either of these two
equivalent conditions holds.
\end{lem}

\begin{proof}
Consider the $\ph$-adic expansions of $f$, $g$, and $fg$ as in Lemma \ref{expprod}. Assume that  $fg(X)$ satisfies the $L_{\ph,\la}$ property, where $\lambda=\nu_p(A_l(X))/l$. Since $f(X)$, $g(X)$, and $\ph(X)$ are monic, $\nu_p(a_0(X))=\nu_p(b_0(X))=0$. Let $\la_1=\mbox{min}\{\nu_p(a_i(X))/i\, | \, i=1, \dots, l_1\}$ and $\la_2= \mbox{min}\{\nu_p(b_i(X))/i\, |\, i=1, \dots, l_2\}$. We claim that $\lambda=\lambda_1=\lambda_2$. Suppose that $\lambda_1<\lambda_2$, and let $i_0=\mbox{min}\{\,i\,|\, \nu_p(a_i(X))/i=\la_1\}$. Consider $c_{i_0}(X)=a_{i_0}(X)b_0(X)+\cdots+a_{0}(X)b_{i_0}(X)$. Then $\nu_p(a_{i_0}(X)b_0(X))=\nu_p(a_{i_0}(X))+\nu_p(b_0(X))=i_0\lambda_1$ and, for $j=1, \dots, i_0$, $$\nu_p(a_{i_0 -j}(X)b_j(X))=\nu_p(a_{i_0-j}(X))+\nu_p(b_j(X))\geq(i_0-j)\lambda_1 +j\lambda_2 > i_0\lambda_1$$ (as $\lambda_2 >\lambda_1$). So $\nu_p(c_{i_0}(X))=i_0\la_1$. Now $\overline{\phi}(X) \nmid (\overline{a_{i_0}}(X)\overline{b_0}(X)/p^{i_0\lambda_1})$ because \linebreak $\nu_p(a_{i_0}(X))=i_0\la_1$, $\overline{\phi}(X) \nmid (\overline{a_{i_0}}(X)/p^{i_0 \lambda_1})$, $\overline{\phi}(X) \nmid \overline{b_0}(X)$, and $\overline{\phi}(X)$ is irreducible. So $\overline{\phi}(X) \nmid (\overline{c_{i_0}}(X)/p^{i_0\lambda_1})$ and, thus, $R_{i_0}(X)/p^{i_0\lambda_1} \not\equiv 0 \;\; (\mbox{mod}\,p)$. Therefore, $\nu_p(R_{i_0}(X))=i_0\lambda_1$. Now, for $i=1, \dots, l_1$, $c_i(X)=a_i(X)b_0(X) + \sum_{j=1}^i a_j(X) b_{i-j}(X)$. As, for $1\leq j\leq i$, $$\nu_p(a_j(X) b_{i-j}(X) \geq j\la_1+(i-j)\la_2 \geq j\la_1 +(i-j)\la_1 =i\la_1,$$ we have $\nu_p(c_i(X)) \geq \nu_p(a_i(X) b_0(X)) \geq i\la_1$. Thus, $\nu_p(Q_{i_0 +1}(X)) \geq (i_0 +1) \lambda_1$. Moreover, since $A_{i_0}(X) = R_{i_0}(X) + Q_{i_0 +1}(X)$, $\nu_p(A_{i_0}(X)) = \nu_p(R_{i_0}(X))=i_0\lambda_1$. So $i_0\lambda_1\geq i_0\lambda$ and $\lambda_1 \geq \lambda$. We now have $$\nu_p(c_l(X))=\nu_p(a_{l_1}(X))+\nu_p(b_{l_2}(X))\geq l_1\lambda_1 +l_2\lambda_2 > l_1\lambda_1+l_2\lambda_1=l\lambda_1\geq l \lambda.$$ It then follows from $\nu_p(c_l(X))=\mbox{min}\{\nu_p(R_l(X)), \nu_p(Q_l(X))\} >l\lambda$ and $A_l(X) = R_l(X)$ that $\nu_p(A_l(X))>l\lambda$. This is a contradiction, as $fg(X)$ satisfies the $L_{\phi, \lambda}$ property. Similarly, assuming $\lambda_1 <\lambda_2$ would yield a contradiction. Hence $\lambda_1=\lambda_2$.

We now show that $\lambda=\lambda_1=\lambda_2$. Let $$i_0=\mbox{max}\{\, i \,|\, \nu_p(a_i(X))/i=\la_1\}\;\; \mbox{and}\;\;
j_0=\mbox{max}\{\, j\,|\,\nu_p(b_j(X))/j=\la_2\}.$$ Consider $c_{i_0+j_0}(X)=a_{0}(X)b_{i_0+j_0}(X)+\dots+a_{i_0}(X) b_{j_0}(X)+\dots+a_{i_0+j_0}(X)b_0(X)$. It then follows that $$\nu_p(c_{i_0+j_0}(X))=\nu_p(a_{i_0}(X)b_{j_0}(X))=i_0\la_1+j_0\la_2=(i_0+j_0)\la_1$$ and that $(i_0+j_0)\la_1(= \nu_p(A_{i_0+j_0}(X)))$ is an integer. Dividing now by $p^{(i_0+j_0)\la_1}$ yields \linebreak $\nu_p(R_{i_0+j_0}(X))=(i_0+j_0)\la_1$ and $\nu_p(A_{i_0+j_0}(X))=(i_0+j_0)\la_1$. But as $fg$ satisfies the $L_{\ph, \la}$ property, we have $\nu_p(A_{i_0+j_0}(X))\geq(i_0+j_0)\la$. So, $\la\le \la_1$. On the other hand, as $l\la=\nu_p(A_l(X))=\nu_p(R_l(X))=\nu_p(c_l(X))=\nu_p(a_{l_1}(X))+\nu_p(b_{l_2}(X))\geq l_1\la_1+l_2\la_2=l\la_1$, we have $\la= \la_1$ as desired.

Now, we have $\nu_p(a_{l_1}(X))+\nu_p(b_{l_2}(X))=l\la=l_1 \la + l_2 \la=l_1\la_1+l_2\la_2$, $\nu_p(a_{l_1}(X))\geq l_1 \la_1$, and $\nu_p(b_{l_2}(X))\geq l_2\la_2$. Thus $\nu_p(a_{l_1}(X))=l_1\la_1=l_1 \la$ and $\nu_p(b_{l_2}(X))=l_2\la_2=l_2 \la$. Hence both $f$ and $g$ satisfy the $L_{\phi, \la}$ property.

Conversely, assume that $f$ and $g$ satisfy the same condition $L_{\ph,\lambda}$. Then  for every $i$,  $\nu_p(a_i(X))\ge i\lambda $ and $\nu_p(b_i(X))\ge i\lambda $. Let $i:=1,\dots, l-1$. Since $c_i(X)=\sum_{k=1}^ia_k(X)b_{i-k}(X)$, $\nu_p(c_i(X))\ge (k+i-k)\lambda\ge i\lambda$. Now, the facts that $A_i(X)=R_i(X)+Q_{i+1}(X)$,\linebreak $\nu_p(R_i(X))\ge i\lambda$, and $\nu_p(Q_{i+1}(X))\ge( i+1)\lambda$ imply that $\nu_p(A_i(X))\ge i\lambda$. As for $i=l$, we have $\nu_p(c_l(X))=\nu_p(a_{l_1}(X))+\nu_p(b_{l_2}(X))=(l_1+l_2)\lambda = l\lambda$. Moreover, since $\overline{\ph}(X)$ does not divide $\overline{c_l}(X)/p^{\nu_p(c_l(X))}$, the remainder $R_l(X)/p^{\nu_p(c_l(X))} \not\equiv 0 \;\; (\mbox{mod}\,p)$  and $$\nu_p(A_l(X))= \nu_p(R_l(X))=\nu_p(c_l(X))=l\lambda.$$ Hence, $fg$ satisfies the same $L_{\ph,\lambda}$ property.

As in the proof of Lemma \ref{slopezero}, since $\nu_p(Q_{ei+1}(X))> ei\la \geq ih$ for each  $i=0,\dots,d$, it \linebreak follows that
$$A_{ei}(X)/p^{ih} \equiv R_{ei}(X)/p^{ih} \equiv  c_{ei}(X)/p^{ih} \equiv
\sum_{k+j=i} a_k(X) b_j(X)/p^{kh+jh} \;\; (\mbox{mod}\,(p,\ph)).$$ But as $a_k(X)/p^{k\la} \, b_j(X)/p^{ie\la} \equiv 0 \;\; (\mbox{mod}\,p)$ when $k$ or $j$ is not a multiple of $e$, we have $A_{ei}(X)/p^{i\la} \equiv c_{ei}(X)/p^{ie\la} \equiv \sum_{k+j=i} a_{ek}(X)/p^{kh}\, b_{ej}(X)/p^{jh} \;\; (\mbox{mod}\,(p, \ph))$. This yields \linebreak $(fg)_{\ph}(Y)={f}_{\ph}(Y){g}_{\ph}(Y)$.
 \end{proof}

\begin{lem}\label{irred}
Let $\overline{\ph}(X)$ be irreducible modulo $p$, $\overline{f}(X)$ congruent to a power of $\overline{\ph}(X)$, and $f(X)$ satisfies the $L_{\ph,\la}$ property. If ${f_{\ph}}(Y)$ is irreducible in $\fph[Y]$, then  $f(X)$ is irreducible in $\z_p[X]$.
\end{lem}

\begin{proof} Suppose that $f(X)$ factors in $\z_p[X]$ as $f(X)=g(X)h(X)$ such that ${g}(X)$ and
${h}(X)$ are nonconstant. Then by Lemma \ref{prod}, both of
${g}(X)$ and ${h}(X)$ satisfy the same condition $L_{\ph,\la}$ and
${f}_{\ph}(Y)={g}_{\ph}(Y){h}_{\ph}(Y)$, where ${g}_{\ph}(Y)$ and
${h}_{\ph}(Y)$ are nonconstant as their respective degrees are positive. So $f_{\ph}(Y)$ is  not
irreducible in $\fph[Y]$, a contradiction.
 \end{proof}

\begin{lem}\label{monicirred}

Let $f(X)$ be irreducible. Then $\overline{f}(X)$ is a power of an irreducible polynomial
$\overline{\phi}(X)$ modulo $p$ and $f(X)$ satisfies the $L_{\phi,\lambda}$ property.

\end{lem}

\begin{proof}
Suppose that  $\overline{f}(X) = \overline{F_1}(X)\,
\overline{F_2}(X)$ modulo $p$ such that $\overline{F_1}(X)$ and $\overline{F_2}(X)$
are nonconstant polynomials coprime modulo $p$. Then by Hensel's lemma, $f(X)=f_1(X)\, f_2(X)$ in $\z_p[X]$ where $f_i$ is a left of $F_i$ for
$i=1,2$. This contradicts the irreducibility of $f(X)$ in $\z_p[X]$. Therefore,
$\overline{f}(X)$ is congruent to a power of an irreducible polynomial $\overline{\phi}(X)$ modulo $p$. Let $\alpha$ be a root of $f(X)$ in some extension of $\mathbb{Q}_p$, $g(X)=X^{n}+b_1X^{n-1}+\cdots+b_{n}\in \z_p[X]$ the minimal polynomial of $\ph(\al)$ over $\q_p$, and $G(X)=g(\ph(X)) \in \z_p[X]$. Let $L$ be the splitting field of $g(X)$, $\z_L$ its ring of
 integers, $\p$ the maximal ideal of $\z_L$, and $u=\nu_{\p}(\ph(\al))$. Using
the symmetric formulas relating the roots and coefficients of $g(X)$ and the fact that the $\p$-adic valuation is invariant
under Galois actions, we have $\nu_{\p}(b_{n})=nu$ and $\nu_{\p}(b_i)\ge iu$ for every
$i=1,\dots,{n}-1$. Setting $\la_g=\nu_{p}(b_{n})/n$, we see that $\la_g=u/e(\p)$ as $\nu_{\p}(b_n)=\nu_p(b_n) e(\p)$, where $e(\p)$ is the ramification index of $\p$ over $p$. So $\nu_{p}(b_i)\ge i(u/e(\p))\ge i\la_g$ for every $i=1,\dots,{n}-1$. It follows that $g(X)$ satisfies the $L_{X,\la_g}$ property, and so does $G(X)$. Furthermore, $\overline{G}(X)$ is congruent to $\overline{\phi}(X)^n$ modulo $p$. Since $G(\al)=0$, $f(X)$ divides
$G(X)$ over $\mathbb{Z}_p$. It then follows from Lemma \ref{prod} that $f(X)$ satisfies the property $L_{\ph,\la_g}$ as well and $\la_g=\la$.
\end{proof}

\begin{lem}\label{valuation}
Let $f(X)$ be irreducible, $\al$ a root of $f(X)$, $\mathbb{K}=\mathbb{Q}_p(\al)$, $\z_\mathbb{K}$ the ring of integers of $\mathbb{K}$, and $\p$ the maximal ideal of $\z_\K$. Assume that $\overline{\ph}(X)$ is irreducible modulo $p$ such that $\overline{f}(X)$ is a power of $\overline{\ph}(X)$. Then there is only one valuation $w$ of $\mathbb{K}$ extending $\nu_p$, and, for every $P(X)\in \z_p[X]$, $w(P(\al)) \geq \nu_p(P(X))$, where equality holds if and only if $\overline{\ph}(X) \nmid (\overline{P}(X)/p^v)$ modulo $p$, where $v=\nu_p(P(X))$.
\end{lem}

\begin{proof}
Since $f(X)$ is irreducible over $\q_p$, Lemma \ref{tool} implies that there is only one valuation $w$ of $\mathbb{K}$ extending $\nu_p$. Consider the $\F_p$-algebra homomorphism $\psi: \F_p[X] \to \z_\mathbb{K}/\p$ defined by $\overline{g}(X) \mapsto \overline{g(\al)}$. Since $\overline{\ph}(X)$ is the minimal polynomial of $\overline{\al}$ over $\F_p$, $\mbox{Ker}\psi$ is the principal ideal of $\F_p[X]$ generated by $\overline{\ph}(X)$. Now for $P(X)\in \z_p[X]$, it is clear that $P(X)/p^v \in \z_p[X]$, where $v=\nu_p(P(X))$. Since, further, $\al \in \z_\mathbb{K}$, $P(\al)/p^v \in \z_\mathbb{K}$. So $w(P(\al)/p^v)\ge 0$ and, thus, $w(P(\al)) \geq w(p^v) = v$. If  $w(P(\al)) > v$, then $P(\al)/p^v \equiv 0 \md{\p}$ and $P(\al)/p^v \in \mbox{Ker}\psi$, which implies that $\overline{\ph}(X)$ divides $\overline{P}(X)/p^v$.
\end{proof}

\section{Proof of Theorem \ref{Main 1}}\label{proofs}

\begin{proof}({\bf Theorem \ref{Main 1}})

It follows from Hensel's lemma that $F(X) = \prod_{i=1}^r f_i(X)$, where $f_i(X) \in \mathbb{Z}_p[X]$ are distinct monic polynomials and, for each $i=1, \dots, r$, $\overline{f_i}(X) = \overline{\phi_i}(X)^{l_i}$ modulo $p$ for some monic irreducible polynomial $\overline{\phi_i}(X)$. By Lemma \ref{tool}, in order to have exactly $r$  valuations of $K$ extending  $\nu_p$, it suffices to show that each $f_i(X)$ is irreducible over $\q_p$. If $l_i = 1$ for some $i$, then $\overline{f_i}(X)$ is irreducible modulo $p$ and, hence, $f_i(X)$ is irreducible over $\q_p$. On the other hand, if $l_i \geq 2$ for some $i$, then let $F(X) = f_i(X) g_i(X)$, where $g_i(X) = \prod_{j \neq i} f_j(X)$. Note that $\overline{\phi_i}(X)$ does not divide $\overline{g_i}(X)$, $\overline{f_i}(X)$ is a power of $\overline{\phi_i}(X)$, and $F(X)$ satisfies the $L_{\phi_i, \lambda_i}$ property. It thus follows from Lemma \ref{slopezero} that $f_i(X)$ satisfies the $L_{\phi_i, \lambda_i}$ property and ${f_i}_{\phi_i}(Y) = c_i F_{\phi_i}(Y)$ for some nonzero $c_i\in \mathbb{F}_{\phi_i}$. Since $F_{\phi_i}(Y)$ is irreducible, so is ${f_i}_{\phi_i}(Y)$. Then, by Lemma \ref{irred}, $f_i(X)$ is irreducible over $\q_p$. We, therefore, conclude that there are exactly $r$ valuations $w_1, \dots, w_r$ of $K$ extending $\nu_p$.

We now calculate the ramification index and residue degree of each $w_i$. To simplify notation, fix some $i = 1, \dots, r$ and set $f(X)=f_i(X)$, $\ph(X)=\ph_i(X)$, $l=l_i$, $m=m_i$, $d=d_i$, and $w=w_i$. Let $\al$ be a root of $f(X)$, $\K=\q_p(\al)$ the local field defined by $f(X)$, $\z_\K$ its ring of integers, and $\p$ the unique prime ideal of $\z_{\K}$. Since $\overline{f}(X)$ is congruent to $\overline{\ph}(X)^{l}$ modulo $p$, the $\ph$-adic expansion of $f(X)$ over $\z_p$ has the form $$f(X)=\ph(X)^{\l}+a_1(X)\ph(X)^{l-1}+\dots+ a_{l-1}(X)\phi(X)+a_{\l}(X),$$ where, for $j=1, \dots, l-1$, $\nu_p(a_j(X))>j\la$, $\la=v/l$, and $v=\nu_p(a_l(X))$. Set $\overline{\al}:=\al$ modulo $\p$. Since $\overline{\ph}(X)$ is the minimal polynomial of $\overline{\al}$ over $\mathbb{F}_p$ and $\overline{\ph}(X)$ does not divide $\overline{a_l}(X)/p^\nu$, it follows from Lemma \ref{valuation} that $v=w(a_l(\al))=l\la$. We now claim that $w(\ph(\al)) =  \la$. Note that as $f(X)$ satisfies the $L_{\phi, \la}$ property, $\nu_p(a_j(X))\geq j\la$ for $j=1, \dots, l-1$. So $w(a_j(\al)) \geq j\la$ (from Lemma \ref{valuation}). Thus, for $j=1, \dots, l-1$, $w(a_j(\al)\ph(\al)^{\l-j})\ge j\la +(\l-j)u$, where $u=w(\ph(\al))$. If $u\neq  \la$, then it follows from the $\ph$-adic expansion of $f(X)$ over $\z_p$ that $w(f(\al))=\mbox{min}(\l u,\l \la)$, which is impossible since $w(f(\al))=\infty$. Thus, $u=\la$ as claimed.

We now show that the value group of $w$ is $\z[\la]$, the subgroup of $\q$ generated by 1 and $\la$. For $P(X) \in \z[X]$, let the $\ph$-adic expansion of $P(X)$ be
$$P(X) = g_0(X) \ph^l(X)+ g_1(X)\ph^{l-1}(X) + \dots + g_{l-1}(X)\ph(X)+ g_l(X).$$
Since $w(\ph(\alpha))=\la$, it follows that, for each $j=0, \dots, l$, $w(g_j(\al)\ph^{l-j}(\al))=n_j+(l-j)\la$, where $n_j=w(g_j(\al))\in\z$. Thus, $w(P(\al)) \in \z[\la]$. As every element of $K$ is of the form $P(\al)/Q(\al)$ for some $P(X), Q(X) \in \z[X]$, the value group of $w$ is thus $\z[\la]$. So, the ramification index $e(w)$ is the index $[\z[\la]:\z]$, which is the order of $\la$ in $\q/\z$ and is equal to $e=l/d$. As $\nu_p$ is discrete, the residue degree is $f(w)=\mbox{deg}(f)/e(w)= lm / e(w) =m d$. The proof is now complete.

\end{proof}

\section{Remarks and Examples}\label{rk and ex}

\subsection{Remarks}\label{remarks}
\begin{enumerate}
\item
It is worthwhile to consider the possibility of a converse of Theorem \ref{Main 1}. Sub-remark (b) below shows that this is not possible. However, sub-remark (a) below further specifies which of the sufficient conditions is really necessary.

\begin{enumerate}
\item
If there exist exactly $r$ valuations $w_1, \dots, w_r$ of $K$ extending $\nu_p$, then $f_i(X)$ is irreducible in $\q_p[X]$ for every $i=1,\dots,r$. So, by Lemma \ref{monicirred}, $f_i$ satisfies the $L_{\ph_i,\la_i}$ property and, thus, so does $F$ (by Lemma \ref{slopezero}). This shows that in order to have exactly $r$  valuations $w_1, \dots, w_r$ of $K$ extending $\nu_p$, the requirement that $F(X)$ satisfies the $L_{\ph_i, \la_i}$ property for each $i=1, \dots, r$ is {\it necessary} for a possible converse of Theorem \ref{Main 1}.

\item
Let $a,b \in \z$ be such that $a\equiv 0 \md{16}$, $a \not\equiv 0 \md{32}$, and $b \equiv 4 \md{32}$. This in particular implies that $\nu_2(a)=4$ and $\nu_2(b)=2$. Assume further that the polynomial $F(X)=X^4+aX+b$ is irreducible over $\z$. Let $\al$ be a complex root of $F(X)$ and $K=\q(\al)$. Then $F(X)\equiv X^4 \md{2}$, $F(X)$ satisfies the $L_{X,1/2}$ property, and $F_X(Y)=(Y+1)^2$ (which is reducible over $\F_X \cong \F_2$). Let $A, B\in \z$ be such that $a=16A$ (so $2 \nmid A$) and $b=4+32B$. Set \linebreak $\epsilon=(\al^2+2\al+2)/4$, and let $l_\epsilon$ be the multiplication-by-$\epsilon$ endomorphism of $K$. Then $P(X)=X^4 -2X^3 +CX^2+DX+E$ is the characteristic polynomial of $l_\epsilon$, where $C=4B+6A+2$, $D=-(4A^2+4A+4B)$, and $E=4B^2-4AB+2A^2$. Since $2$ divides all the coefficients of $P(X)$ except the leading one and $2^2 \nmid E$, $P(X)$ is 2-Eisenstein and hence irreducible in $\z[X]$. So $\epsilon\in\z_K$ and $K=\q(\epsilon)$. \linebreak Now $P(X)\equiv X^4 \md{2}$ and $M(X)=(P(X)-X^4)/2$. So,

$M(X) =-X^3+(2B+3A+1)X^2-(2A^2+2A+2B)X+(2B^2-2AB+A^2)$. \linebreak Since $2\nmid A$, $A\equiv 1 \md2$ and  $M(X) \equiv -X^3+1 \md{2}$. So $X \nmid \overline{M}(X)$ modulo 2. Thus, by Dedekind's criterion, 2 does not divide $\mbox{ind}(\epsilon)$ and, therefore there is only one valuation $w$ of $K$ extending $\nu_2$ with ramification index $e(w)=4$ and residue degree $f(w)=1$. So, the conclusion of Theorem \ref{Main 1} holds despite that some $F_{\ph_i}(Y)$ is not irreducible. So the requirement that $F_{\ph_i}(Y)$ be irreducible for all $i=1, \dots, r$ is {\it not necessary} for a possible converse of Theorem \ref{Main 1}.

\end{enumerate}

\item
For every $i=1, \dots, r$, consider the $\ph_i$-adic expansion $$F(X)=a_{i,0}(X)\ph_i(X)^{L_i} + \dots + a_{i, L_i}(X).$$ Assume that $F(X)$ satisfies the
$L_{\ph_i,\la_i}$ property for every $i$, where $\la_i=\nu_{p}(a_{i, L_i}(X))/\l_i$. If $\l_i$ is the order of $\la_i$ in the group $\q/\z$, then $\nu_{p}(a_{i, L_i}(X))$ and $l_i$ are coprime. So $d_i=1$ and, thus, $F_{\ph_i}(Y)$ is irreducible (as it is linear). It follows that Theorem \ref{Main 1} generalizes  \cite[Theorem 1.1]{KK2} in the context of number fields.

\item
We give here an easier proof of \cite[Theorem 1.1]{KK1} in the case that the ring is $\z$. Assume that for every $i=1,\dots,r$, $\p_i=(p,\ph_i(\al))$, $f(\p_i)=\mbox{deg}(\ph_i)$, and $e(\p_i)= l_i$. We show that $p \nmid \mbox{ind}(\al)$.

On the one hand, if every $l_i=1$, then (by Dedekind's criterion) $p \nmid \mbox{ind}(\al)$. On the other hand, if $l_i\ge 2$ for some $i=1, \dots, r$, then it follows that $h_i:=\nu_{\p_i}(\ph_i(\al))=1$ (because $\ph_i(\al)$ is a generator of $\p_i$). As $e(\p_i)= l_i$, $d_i=1$. We need now to show that the conditions of Dedekind's criterion are satisfied in this case.
 Let $$F(X)=a_0(X)\ph_i^{L_i}(X)+\dots+a_{s_i}(X)\ph_i^{l_i}(X)+\dots+a_{L_i}(X)$$ be the $\ph_i$-adic expansion of $F(X)$. Let $Q_i(X)=a_0(X)\ph_i^{L_i-l_i}(X)+\dots+a_{s_i}(X)$ and $R_i(X)=a_{s_i+1}(X)\ph_i^{l_i-1}(X)+\dots+a_{L_i}(X)$. Then $\displaystyle{F(X)=Q_i(X)\ph_i^{l_i}(X)+R_i(X)}$.
 As $F(X)\equiv \prod_{i=1}^r\ph_i^{l_i}(X)\md{p}$, it follows that $R_i(X)\equiv 0 \md{p}$ and \linebreak $Q_i(X) \equiv \prod_{j\neq i}\ph_j^{l_j}(X) \md{p}$. Thus, $\nu_p(a_{s_i+k}(X))\ge 1$ for every $k=1,\dots,l_i$, and there exists some $H_i(X)\in\z[X]$ such that ${Q_i}(X)=\prod_{j\neq i}\ph_j^{l_j}(X)+pH_i(X)$. If we let $M(X)= (F(X)-\prod_{i=1}^r\ph_i^{l_i}(X))/p$, then we have $$M(X)=H_i(X)\ph_i^{l_i}(X)+(a_{s_i+1}(X)\ph_i^{l_i-1}(X)+
\dots+a_{L_i}(X))/p.$$ As $d_i=h_i=1$, $\nu_p(a_{L_i}(X))=1$. So $p$ does not divide $a_{L_i}(X)/p$ and, therefore, $\overline{\ph_i}(X)$ does not divide $\overline{M}(X)$. Hence, $p \nmid \mbox{ind}(\al)$.

\end{enumerate}

\subsection{Examples}\label{examples}

The following two examples are to emphasize the advantage of the test of Theorem \ref{Main 1} over the test of K-K's result. The third example presents a situation where Theorem \ref{Main 1} is not applicable, yet remark (1-a) of Subsection \ref{remarks} turns out to be useful in some sense.\\

\noindent{\bf Example 1:}

Let $F(X)=X^7+3X^5+18X^4+9X^3+6X^2+48X+24$. It is clear that $F(X)$ is \linebreak $3$-Eisenstein and, so, is irreducible over $\z$. Let $K=\q(\al)$, where $\al$ is a complex root of $F(X)$. We want to see how $\nu_2$ extends to $K$. Note first that $F(X)\equiv X^3(X^2+X+1)^2 \md 2$. Let $\phi_1(X)=X$ and $\phi_2(X)=X^2+X+1$. Then
\begin{align*}
F(X)&= \phi_1(X)^7+3\phi_1(X)^5+18\phi_1(X)^4+9\phi_1(X)^3+6\phi_1(X)^2+48\phi_1(X)+24 \\
&=(X-3)\phi_2(X)^3+(6X+17)\phi_2(X)^2-(28X+14)\phi_2(X)+(58X+24).
\end{align*}
It can be easily checked that $F(X)$ satisfies the $L_{\ph_i,\la_i}$ property for $i=1, 2$, where $\la_1=1$ and $\la_2=1/2$. For $i=1$, we have $\nu_2(a_{1,7}(X))=\nu_2(24)=3$, $s_1 =4$, and so $d_1=3$.  Moreover $F_{\ph_1}(Y)=Y^3+Y^2+1$, which is irreducible over $\F_{\ph_1} \cong \F_{2}$. For $i=2$, as $\nu_2(a_{2,4}(X))=\nu_2(58+24)=1$, $d_2=1$ and thus $F_{\ph_2}(Y)$  is linear and, thus, irreducible over $\F_{\ph_2} \cong \F_4$. It now follows from Theorem \ref{Main 1} that there are exactly two valuations $w_1$ and $w_2$ extending $\nu_2$ to $K$ with $e(w_1)=1$, $f(w_1)=3$, $e(w_2)=2$, and $f(w_2)=2$.\\

\noindent{\bf Example 2:}

Let $F(X)=X^6+36X^3+48$. It is clear that $F(X)$ is $3$-Eisenstein and, so, is irreducible over $\z$. Let $K=\q(\al)$ where $\al$ is a complex root of $F(X)$. We want to see how 2 factors in $\z_K$. We can see that $F(X)\equiv X^6 \md{2}$. Let $\ph(X)=X$; then $F(X)=\ph(X)^6+36\ph(X)^3+48$. It can be easily checked that $F(X)$ satisfies the $L_{\ph, \la}$ property, where $\la=2/3$. Since $\nu_2(a_6(X))=\nu_2(48)=4$ and $s=0$, $d=2$. It can be checked that $F_\ph(Y)=Y^2+Y+1$, which is irreducible over $\F_\ph \cong \F_2$. It thus follows from Theorem \ref{Main 1} that there is only one valuation $w$ of $K$ extending $\nu_2$, $e(w)=3$, and $f(w)=2$.\\

\noindent{\bf Example 3:}

Let $F(X)=X^9+48X^7+6X^6+24X^5+12X^4+3X^3+18X^2+6X+12$. It is clear that $F(X)$ is 3-Eisenstein and, so, is irreducible over $\z$. We can see that $$F(X)\equiv X^3(X-1)^2(X^2+X+1)^2 \md 2.$$ Letting $\ph(X)=X$, we see that $$F(X)=\ph(X)^9+48\ph(X)^7+6\ph(X)^6+24\ph(X)^5+12\ph(X)^4+3\ph(X)^3+18\ph(X)^2+6\ph(X)+12,$$ $s=6$ and, thus, $\la=2/3$. As $\nu_2(a_8(X))=1$, we have $\nu_2(a_8(X))/2=1/2<2/3$ and, therefore, $F(X)$ does not satisfy the $L_{\ph,\la}$ property. So, in fact, Theorem \ref{Main 1} cannot be used here. But at least we know from remark (1-a) of Subsection \ref{remarks} that the number of distinct valuations of $K$ extending $\nu_2$ is certainly more than 3.

\end{document}